\documentclass[12pt]{amsart}
\usepackage{amscd}
\usepackage{amssymb}

\allowdisplaybreaks[1]

\newtheorem{thm}{Theorem}[section]
\newtheorem{rema}[thm]{Remark}
\newtheorem{cor}[thm]{Corollary}
\newtheorem{lem}[thm]{Lemma}
\newtheorem{prop}[thm]{Proposition}

\newtheorem{defn}[thm]{Definition}
\newtheorem{ex}[thm]{Example}

\theoremstyle{definition}

\numberwithin{equation}{section}

\newcommand{\R}{\mathbb R}

\newcommand{\tr}{{\rm tr}}

\newcommand{\resumename}{R\'esum\'e}

\newcommand{\eqKu}{
\stackrel{^{\text{\tiny deg
}}}{=}}

\begin{document}

\date{\today}

\title[Pinczon algebras]{Quadratic and Pinczon algebras}
\author[D. Arnal]{Didier Arnal}
\author[W. Bakbrahem]{Wissem Bakbrahem}
\author[M. Selmi]{Mohamed Selmi}
\address{
Institut de Math\'ematiques de Bourgogne\\
UMR CNRS 5584\\
Universit\'e de Bourgogne-Franche Comt\'e\\
U.F.R. Sciences et Techniques
B.P. 47870\\
F-21078 Dijon Cedex\\France}
\email{Didier.Arnal@u-bourgogne.fr}
\address{
Universit\'e de Sousse, Laboratoire de Math\'ematique Physique, Fonctions sp\'eciales et Applications\\
Ecole Sup\'erieure des Sciences et de Technologie de Hammam Sousse\\
Rue Lamine Abassi \\
4011 H.Sousse \\Tuni\-sie}
\email{wissem-essths@hotmail.fr}
\address{
Universit\'e de Sousse, Laboratoire de Math\'ematique Physique, Fonctions sp\'eciales et Applications\\
Ecole Sup\'erieure des Sciences et de Technologie de Hammam Sousse\\
Rue Lamine Abassi \\
4011 H.Sousse \\Tuni\-sie}
\email{mohamed.selmi@fss.rnu.tn}



\begin{abstract}
Given a symmetric non degenerated bilinear form $b$ on a vector space $V$, G. Pinczon and R. Ushirobira defined a bracket $\{~,~\}$ on the space of multilinear skewsymmetric forms on $V$. With this bracket, the quadratic Lie algebra structure equation on $(V,b)$ becomes simply $\{\Omega,\Omega\}=0$.

We characterize similarly quadratic associative, commutative or pre-Lie structures on $(V,b)$ by the same equation $\{\Omega,\Omega\}=0$, but on different spaces of forms. These definitions extend to quadratic up to homotopy algebras and allows to describe the corresponding cohomologies.\\

\end{abstract}

\keywords{Quadratic algebras, Cohomology, Graded Lie algebras}
\subjclass[2000]{17A45, 16E40, 17BXX, 17B63}


\maketitle


\section{Introduction}


In \cite{PU,DPU} (see also \cite{D,MPU}), Georges Pinczon and Rosane Ushirobira introduced what they called a Poisson bracket on the space of skewsymmetric forms on a finite dimensional vector space $V$, equipped with a symmetric, non degenerated, bilinear form $b$. If $(e_i)$ is a basis in $V$ and $(e'_j)$ the basis defined by the relations $b(e'_j,e_i)=\delta_{ij}$, the bracket is:
$$
\{\alpha,\beta\}=\sum_i\iota_{e_i}\alpha\wedge\iota_{e'_i}\beta.
$$
Especially, if $\alpha$ is a $(k+1)$-form, and $\beta$ a $(k'+1)$, then $\{\alpha,\beta\}$ is a $(k+k')$-form. Shifting the degree on $V$ by -1, we replace $V$ by $V[1]$, the skew-symmetric bilinear forms on $V$ becomes symmetric forms on $V[1]$ and the bracket $\{~,~\}$ a Lie bracket.\\

In fact, the authors proved that a structure of quadratic Lie algebra $(V,[~,~],b)$ on $V$ is completely characterized by a 3-form $I$, such that $\{I,I\}=0$. The relation between the Lie bracket and $I$ is simply $I(x,y,z)=b([x,y],z)$, and the equation $\{I,I\}=0$ is the structure equation. A direct consequence of this construction is the existence of a cohomology on the space of forms, given by:
$$
d\alpha=\{\alpha,I\}.
$$
This cohomology characterizes the problem of definition and deformation of quadratic Lie algebra structure on $(V,b)$.\\

In this paper, we first generalize this construction, defining the Pinczon bracket on the space $\mathcal C(V)$ of cyclic multilinear forms on the shifted space $V[1]$. 
A Pinczon bracket on $\mathcal C(V)$ is a Lie bracket 
$(\Omega,\Omega')\mapsto\{\Omega,\Omega'\}$ such that $\{\Omega,\Omega'\}$ is $(k+k')$-linear, if $\Omega$ is $(k+1)$-linear, and $\Omega'$ $(k'+1)$-linear, and, for each linear form $\alpha$, $\{\alpha,~\}$ is a derivation of the cyclic product.\\

As in \cite{PU}, there is a one-to-one correspondence between the set of Pinczon brackets on $\mathcal C(V)$ and the 
set of non degenerated, bilinear symmetric forms $b$ on $V$, and the correspondence is given through the Pinczon-Ushirobira formula:
$$
\{\Omega,\Omega'\}=\sum_i\iota_{e_i}\Omega\odot\iota_{e'_i}\Omega'.
$$ 

On the other hand, it is well known that the space $\bigotimes^+ V[1]$ is a cogebra for the comultiplication given by the deconcatenation map:
$$
\Delta(x_1\otimes\ldots\otimes x_k)=\sum_{i=1}^{k-1}(x_1\otimes\ldots\otimes x_i)\bigotimes(x_{i+1}\otimes\ldots\otimes x_k).
$$
Moreover, the coderivations of $\Delta$ are characterized by their Taylor series $(Q_k)$ where $Q_k:\bigotimes^kV[1]\longrightarrow V[1]$, and the bracket $[Q,Q']$ of two such coderivation is still a coderivation.\\

Suppose now there is a symmetric non degenerated form $b$ on $V$. Denotes $B$ the corresponding form, but on $V[1]$. T
here is a bijective map between $\mathcal C(V)$ and the space $\mathcal D_B$ of $B$-quadratic coderivation $Q$, given by the formula:
$$
\Omega_Q(x_1,\ldots,x_{k+1})=B(Q(x_1,\ldots,x_k),x_{k+1}).
$$
This map is an isomorphism of Lie algebra: $\{\Omega_Q,\Omega_{Q'}\}=\Omega_{[Q,Q']}$.\\

With this construction, the notion of quadratic associative algebra, respectively quadratic associative algebra up to homotopy does coincide with the notion of Pinczon algebra structure on $\mathcal C(V)$. This gives also an explicit way to refind the Hochschild cohomology defined by the algebra structure.\\

Moreover, the subspace $\mathcal C_{vsp}(V)$ of cyclic, vanishing on shuffle products forms is a Lie subalgebra of $\mathcal C(V)$. The restriction to this subalgebra of the above construction gives us the notion of quadratic commutative algebra (up to homotopy): it is a Pinczon algebra structure on $\mathcal C_{vsp}(V)$. Similarly, one refinds the Harrison cohomology associated to commutative algebras.\\

A natural quotient of $(\mathcal C(V),\{~,~\})$ is the Lie algebra $(\mathcal S,\{~,~\}|)$ of totally symmetric multilinear forms on $V[1]$. This allows us to refind the notion of quadratic Lie algebra (up to homotopy), and the corresponding Chevalley cohomology.\\

Considering now bi-symmetric mulitlinear forms on $V[1]$: {\sl i.e.} separatly symmetric in their two last variables and in all their other variables, and extending canonically the Pinczon bracket to a bracket $\{~,~\}^+$ on this space of forms, we can define similarly quadratic pre-Lie algebra structures on $(V,b)$, and the corresponding pre-Lie cohomology.\\

Finally, we study the natural example of the space of $n\times n$ matrices. An unpublished preprint (\cite{A}) contains a part of these results.\\


\section{Cyclic forms}

\subsection{Koszul's rule}

\

In this paper, $V$ is a finite dimensional graded vector space, on a charateristic 0 field $\mathbb K$. Denote $|x|$ the degree of a vector $x$ in $V$.\\

First recall the sign rule due to Koszul (see \cite{Ko}). For any relation between quantities using letters representing homogeneous objects, in different orderings, it is always understood that there is an implicit + sign in front of the first term. For each other term, if $\sigma$ is the permutation of the letters between the first quantity and the considered term, there is the implicit sign $\varepsilon_{|letters|}(\sigma)$ (the sign of $\sigma$, taking into account only positions of the odd degree letters) in front of it. 

As usual, $V[1]$ is the space $V$ with a shifted degree. If $x$ is homogeneous, its degree in $V[1]$ is $\deg(x)=|x|-1$. Note simply $x$ for $\deg(x)$ when no confusion is possible. Very generally, we use a small letter for each mapping defined on $V$, and capital letter for the `corresponding' mapping, defined on $V[1]$. Let us define now these corresponding mappings.\\

For any $k$, define a `$k$-cocycle' $\eta$ by putting $\eta(x_1,\ldots,x_k)=(-1)^{\sum_{j\leq k}(k-j)x_j}$. Then, if $\varepsilon(\sigma)$ is the sign of the permutation $\sigma$ in $\mathfrak S_k$, 
$$
\eta(x_{\sigma(1)},\ldots,x_{\sigma(k)})\eta(x_1,\ldots,x_k)=\varepsilon(\sigma)\varepsilon_{|x|}(\sigma)\varepsilon_x(\sigma).
$$

If $q$ is a $k$-linear mapping from $V^k$ into a graded vector space $W$, define the associated $k$-linear mapping from $V[1]^k$ into $W[1]$, with $\deg(Q)=|q|+k-1$, by
$$
Q(x_1,\ldots,x_k)=\eta(x_1,\ldots,x_k)q(x_1,\ldots,x_k).
$$
Prefering to keep the 0 degree for scalars, if $\omega$ is a $k$-linear form on $V$, the same formula associates to $\omega$ a form $\Omega$, on $V[1]$, with degree $\deg(\Omega)=|\omega|+k$.\\ 

This shift of degree modifies the symmetry properties of these mappings. For instance, if $q$ is $\sigma$-invariant, then $Q$ is $\sigma$ skew-invariant.\\

\subsection{Pinczon bracket.}

\

It is defined on cyclic forms.

\begin{defn}
Let $\Omega$ be a $(k+1)$-linear form, $\Omega$ is a {\it cyclic} form on $V[1]$, if it satisfies:
$$
\Omega(x_{k+1},x_1,\ldots,x_k)=\Omega(x_1,\ldots,x_{k+1}).
$$
Denote $\mathcal C(V)$ the space of cyclic forms on $V[1]$.\\
\end{defn}

Let $\sigma\in\mathfrak S_k$ and $\Omega$ $k$-linear. Put 
$
(\Omega^\sigma)(x_1,\ldots,x_{k})= \Omega(x_{\sigma^{-1}(1)},\ldots,x_{\sigma^{-1}(k)}).
$
Denote $Cycl$ the subgroup of $\mathfrak S_{k}$ generated by the cycle $(1,2,\ldots,k)$. Put 
$$
\Omega^{Cycl}= \sum_{\tau\in Cycl}\Omega^\tau, \quad\text{ and }\quad
A\odot B=(A\otimes B)^{Cycl}.
$$

It is easy to prove that the so defined cyclic product is commutative, but non associative.\\

\begin{defn}
A Pinczon bracket $\{~,~~\}$ on the space $\mathcal C(V)$ of cyclic multilinear forms on a graded space $V[1]$ is a bilinear map  
such that
\begin{itemize}
\item[1.] If $\mathcal C^k$ is the space of $k$-linear cyclic forms, $\{\mathcal C^{k+1},\mathcal C^{k'+1}\}\subset\mathcal C^{k+k'}$,
\item[2.] $\mathcal C(V)$, equipped with $\{~,~\}$ is a graded Lie algebra, with center $\mathcal C^0$,
\item[3.] for any linear form $\alpha$, $\{\alpha,~\}$ is a derivation: if $\beta_1,\ldots,\beta_k$ are linear,
$$
\{\alpha,(\beta_1\otimes\ldots\otimes\beta_k)^{Cycl}\}=\sum_j\{\alpha,\beta_j\}(\beta_1\otimes\ldots\hat{\beta_j}\ldots\otimes\beta_k)^{Cycl},
$$
\end{itemize}
\end{defn}

Now,

\begin{prop}
\begin{itemize}
\item[1.] There is a bijective map between the set $\mathcal P$ of Pinczon bracket on $\mathcal C(V)$ and the set $\mathcal B$ of degree 0, symmetric, non degenerated bilinear forms $b$ on $V$,
\item[2.] Let $b$ be in $\mathcal B$, and $(e_i)$ a basis for $V$, if $(e'_i)$ is the basis defined by $b(e_i,e'_j)=\delta_{ij}$, then the Pinczon bracket associated to $b$ is:
$$
\{\Omega,\Omega'\}=\sum_i\iota_{e_i}\Omega\odot\iota_{e'_i}\Omega'.
$$
\end{itemize}
\end{prop}

\begin{proof}

Let $\{~,~\}$ be a Pinczon bracket on $\mathcal C(V)$, then $\{\mathcal C^0,\mathcal C(V)\}=0$, and $\{\mathcal C^1,\mathcal C^1\}\subset\mathbb K$.
The bracket defines a degree 0 bilinear, antisymmetric form $B^\star$ on $\mathcal C^1=(V[1])^\star$. As above this defines a symmetric bilinear form $b^\star$ on the space $V^\star=(V[1])^\star[1]$.

Suppose that for some $\alpha$ in $\mathcal C^1$, $B^\star(\alpha,\mathcal C^1)=\{\alpha,\mathcal C^1\}=0$. Since $\{\alpha,~\}$ is a derivation, $\{\alpha,\mathcal C^k\}=0$, and $\alpha$ is a central element. Thus $\alpha=0$, $b^\star$ is non degenerated and allows to identify $V^\star$ and $V$ and $b^\star$ to a non degenerated bilinear symmetric form $b$ on $V$.\\

Let $B^\star$ be a skewsymmetric bilinear form on $(V[1])^\star$. For any basis $(e_i)$ of $V[1]$ there are vectors $e'_i$ such that:
$$
B^\star(\alpha,\beta)=\sum_i\iota_{e_i}\alpha\otimes \iota_{e'_i}\beta=\sum_i\alpha(e_i)\beta(e'_i)=\sum_i\iota_{e_i}\alpha\odot\iota_{e'_i}\beta.
$$

Coming back to $V^\star$ this means, if all the objects are homogeneous,
$$\aligned
b^\star(\alpha,\beta)&=\sum_i(-1)^{|\beta|}\alpha(e_i)\beta(e'_i).
\endaligned
$$
Identify $V^\star$ to $V$ by defining, for any $\gamma$ in $V^\star$, the vector $x_\gamma$ in $V$ such that $\alpha(x_\gamma)=b^\star(\alpha,\gamma)$, for any $\alpha$ in $V^\star$. 
Then $e'_i=x_{\epsilon_i}$, where $(\epsilon_i)$ is the dual basis of $(e_i)$. Therefore $(e'_i)$ is a basis for $V$, and 
$
b(e'_j,e_i)
=\delta_{ij}$.

Consider now the bracket:
$$
\{\Omega,\Omega'\}_P=\sum_i\iota_{e_i}\Omega\odot \iota_{e'_i}\Omega',
$$
where $\Omega$ is a $k+1$-linear cyclic form and $\Omega'$ a $k'+1$-linear one. In a following section, we shall prove this bracket defines a graded Lie algebra structure on $\mathcal C(V)$. It is clear that, for any $\alpha$ and $\beta_j$ linear,
$$
\{\alpha,(\beta_1\otimes\ldots\otimes\beta_k)^{Cycl}\}_P=\sum_j\{\alpha,\beta_j\}_P(\beta_1\otimes\ldots\hat{\beta_j}\ldots\otimes\beta_k)^{Cycl}.
$$
Moreover the center of the Lie algebra $(\mathcal C(V),\{~,~\}_P)$ is $\mathcal C^0$. In other word, $\{~,~\}_P$ is a Pinczon bracket.\\

If $k+k'\leq0$, $\{\Omega,\Omega'\}=\{\Omega,\Omega'\}_P$. Suppose by induction this relation holds for $k+k'<N$ and consider $\Omega$ and $\Omega'$ such that $k+k'=N$. For any $i$,
$$\aligned
\{\epsilon_i,\{\Omega,\Omega'\}\}
&=-\{\Omega,\{\Omega',\epsilon_i\}_P\}_P-\{\Omega',\{\epsilon_i,\Omega\}_P\}_P
=\{\epsilon_i,\{\Omega,\Omega'\}_P\}_P
=\iota_{e'_i}\{\Omega,\Omega'\}_P
\endaligned
$$

On the other hand, if $\{\Omega,\Omega'\}=\sum_\beta(\beta_1\otimes\ldots\otimes\beta_{k+k'})^{Cycl}$,
$$\aligned
\{\epsilon_i,\{\Omega,\Omega'\}\}&=\sum_{\beta,j}\beta_j(e'_i)(\beta_1\otimes\ldots\hat{\beta_j}\ldots\otimes\beta_{k+k'})^{Cycl}
=\iota_{e'_i}\{\Omega,\Omega'\}.
\endaligned
$$
This proves the existence and unicity, and gives the form of the Pinczon bracket associated to the symmetric, non degenerated bilinear form $b$ on $V$.

\end{proof}



\section{Codifferential}


\

\subsection{General construction}

\

The deconcatenation $\Delta$ is a natural comultiplication in the tensor algebra $\bigotimes^+V[1]$:
$$
\Delta(x_1\otimes\ldots\otimes x_k)=\sum_{r=1}^{k-1}(x_1\otimes\ldots\otimes x_r)\bigotimes(x_{r+1}\otimes\ldots\otimes x_k).
$$
The space $\mathcal D$ of coderivations of $\Delta$ is a natural graded Lie algebra for the commutator. It is well known (see for instance \cite{K,LM}) that any multilinear mapping $Q$ can be extended in an unique way into a coderivation $D_Q$ of $\Delta$, and conversely any coderivation $D$ has an unique form $D=\sum_k D_{Q_k}$. Thus the space of multilinear mappings is a graded Lie algebra for the bracket $[Q,Q']=Q\circ Q'-Q'\circ Q$, where, if $x_{[a,b]}=x_a\otimes x_{a+1}\otimes\dots\otimes x_b$,
$$
Q\circ Q'(x_{[1,k+k'-1]})=\sum_{r=0}^{k-1}Q(x_{[1,r]},Q'(x_{[r+1,r+k']}),x_{[r+k'+1,k+k'-1]}).
$$

\subsection{Relation with the Pinczon bracket}

\

Consider a vector space $V$ equipped with a symmetric, non degenerated bilinear form $b$, with degree 0, denote $B$ the associated form on $V[1]$. 
For any linear map $Q:V[1]^k\rightarrow V[1]$, define the $(k+1)$-linear form:
$$
\Omega_Q(x_1,\ldots,x_{k+1})=B(Q(x_1,\ldots,x_k),x_{k+1}),
$$
and let us say that $Q$ is $B$-quadratic if $\Omega_Q$ is cyclic. Denote $\mathcal D_B$ the space of $B$-quadratic multilinear maps
.\\

The fundamental examples of cyclic maps are mappings associated to a Lie bracket or an associative multiplication on $V$. More precisely, if $(x,y)\mapsto q(x,y)$ is any internal law, with degree 0, and $Q(x,y)=(-1)^xq(x,y)$, then $\Omega_Q$ is cyclic if and only if:
$$\aligned
b(q(x,y),z)
=b(x,q(y,z)),
\endaligned
$$
if and only if $(V,q,b)$ is 
a quadratic algebra.\\

Remark that if $q$ is a Lie bracket, then $\Omega_Q$ is symmetric. Now, if $q$ is a commutative (and associative) product, then 
$\Omega_Q$ is vanishing on the image of the shuffle product on the 2 first variables:
$$
sh_{(1,1)}(x_1\otimes x_2)=(x_1\otimes x_2)+(x_2\otimes x_1).
$$

\begin{prop}

The space $\mathcal D_B$ of $B$-quadratic maps $Q$ is a Lie subalgebra of $\mathcal D$.

The space $\mathcal C(V)$ of cyclic forms, equipped with the Pinczon bracket is a graded Lie algebra, isomorphic to $\mathcal D_B$. \\
\end{prop}

\begin{proof}
For any sequence $I=\{i_1,\ldots,i_k\}$ of indices, denote $x_I$ the tensor $x_{i_1}\otimes\ldots\otimes x_{i_k}$.
 
Suppose $Q$ is $k$-linear, $Q'$ $k'$-linear, both $B$-quadratic. Thus:
$$\aligned
B(Q(x_{[1,k]}),Q'(x_{]k,k+k']}))&=-B(Q'(Q(x_{[1,k]}),x_{]k,k+k'-1]}),x_{k+k'})\\
&=B(Q(Q'(x_{]k,k+k']}),x_{[1,k-1]}),x_k).
\endaligned
$$
Therefore:
$$\aligned
&B(Q\circ Q'(x_{[2,k+k']}),x_1)=\sum_{1\leq r}B(Q(x_{[2,r]},Q'(x_{]r,r+k']}),x_{]r+k',k+k']}),x_1)\\
&=\sum_{1\leq r<k}B(Q(x_{[1,r]},Q'(x_{]r,r+k']}),x_{]r+k',k+k'-1]}),x_{k+k'})
-B(Q'(Q(x_{[1,k]}),x_{]k,k+k'-1]}),x_{k+k'})\\
\endaligned
$$
Or
$$\aligned
B([Q,Q'](x_{[2,k+k']}),x_1)=
B([Q,Q'](x_{[1,k+k'-1]}),x_{k+k'})
\endaligned
$$
This means $[Q,Q']$ is $B$-quadratic, $\mathcal D_B$ is a Lie subalgebra of $\mathcal D$.\\

Now $\Omega_Q$ (resp. $\Omega_{Q'}$) is a $k+1$-linear (resp. $k'+1$-linear) cyclic form, and
$$\aligned
\{\Omega_Q,\Omega_{Q'}\}&(x_1,\ldots,x_{k+k'})=\left(\sum_i\iota_{e_i}\Omega_Q\otimes\iota_{e'_i}\Omega_{Q'}\right)^{Cycl}(x_1,\ldots,x_{k+k'})\\
&=\sum_{\sigma\in Cycl}\sum_i B(Q(x_{\sigma^{-1}(1)},\ldots,x_{\sigma^{-1}(k)}),e_i)
B(Q'(x_{\sigma^{-1}(k+1)},\ldots,x_{\sigma^{-1}(k+k')}),e'_i)\\
&=\sum_{\sigma\in Cycl}B(Q(x_{\sigma^{-1}([1,k])}),Q'(x_{\sigma^{-1}([k+1,k+k'])})).
\endaligned
$$

Consider a term in this sum, such that $k+k'$ belongs to $\sigma^{-1}([1,k])$. This term is:
$$
B(Q(x_{[r+k'+1,k+k']},x_{[1,r]}),Q'(x_{[r+1,r+k']}))=B(Q(x_{[1,r]},Q'(x_{[r+1,r+k']}),x_{[r+k'+1,k+k'-1]}),x_{k+k'}).
$$
The sum of all these terms is just $B((Q\circ Q')(x_{[1,k+k'-1]},x_{k+k'}))$.

Similarly, if $k+k'$ is in $\sigma^{-1}([k+1,k+k'])$, we get:
$$
B(Q(x_{[r+1,r+k']}),Q'(x_{[r+k+1,k+k']},x_{[1,r]}))=-B(Q'(x_{[1,r]},Q(x_{[r+1,r+k]}),x_{[r+k+1,k+k'-1]}),x_{k+k'}),
$$
and the corresponding sum is $-B((Q'\circ Q)(x_{[1,k+k'-1]},x_{k+k'}))$.

This proves:
$$
\{\Omega_Q,\Omega_{Q'}\}=\Omega_{[Q,Q']},
$$
and the proposition, since $Q\mapsto\Omega_Q$ is bijective.\\
\end{proof}

Let us now study in a more detailled way different cases, when $Q$ corresponds to an associative, or a commutative or a Lie, or a pre-Lie structure, or to an up to homotopy such structure.\\


\section{Associative Pinczon algebras}


\

\subsection{Associative quadratic algebras}

\

Suppose now $q$ is a degree 0 associative product, and $b$ is invariant, then the associated coderivation $Q$ of $\Delta$, with degree 1 is the Bar resolution of the associative algebra $(V,q)$. The associativity of $q$ is equivalent to the relation $[Q,Q]=0$.

More generally, a structure of $A_\infty$ algebra (or associative algebra up to homotopy) on the space $V$ is a degree 1 coderivation $Q$ of $\Delta$ on $\otimes^+V[1]$, such that $[Q,Q]=0$. With this last relation, the Pinczon coboundary $d_P:\Lambda\mapsto\{\Omega_Q,\Lambda\}$ is a degree 1 differential on the (graded) Lie algebra $\mathcal C(V)$. The corresponding cohomology is the Pinczon cohomology of cyclic forms. Write also $\Omega_{d_PQ}=d_P\Omega_Q$.\\

\begin{defn}
An associative Pinczon algebra $(\mathcal C(V),\{~,~\},\Omega)$ is a Pinczon bracket $\{~,~\}$ on $\mathcal C(V)$, and a degree 3 form $\Omega\in \mathcal C(V)$, such that $\{\Omega,\Omega\}=0$.\\
\end{defn}

If $\Omega$ is trilinear, then an associative Pinczon algebra is simply a quadratic associative algebra $(V,b,q)$, where $b$ is the symmetric non degenerated form coming from the Pinczon bracket, and $q$ the bilinear mapping associated to $Q$ such that $\Omega=\Omega_Q$.\\

\begin{prop}
Let $(\mathcal C(V),\{~,~\},\Omega)$ be an associative Pinczon algebra. Write $\Omega=\Omega_Q$, $Q=\sum_kQ_k$, with $Q_k:\otimes^k V[1]\rightarrow V[1]$. Then $Q$ is a structure of $A_\infty$ algebra on $V$, and each $Q_k$ is $B$-quadratic for the bilinear form $B$ coming from the bracket.

Conversely, if $(V,b)$ is a vector space with a non degenerated symmetric bilinear form, any $B$-quadratic structure $Q$ of $A_\infty$ algebra on $V$ defines an unique structure of associative Pinczon algebra on $V$.\\
\end{prop}


\subsection{Bimodules and Hochschild cohomology}

\

Suppose $(V,q)$ is an associative algebra and $M$ a bi-module. Then the Hochschild cohomology with value in $M$ is a part of the Pinczon cohomology of a natural Pinczon algebra.\\

Consider first the semidirect product $W=V\rtimes M$, that is the vector space $V\times M$, equipped with the associative product $q_W((x,a),(y,b))=(q(x,y),(x\cdot b+a\cdot y))$. The dual $W^\star=V^\star\times M^\star$ is now a $(W,q_W)$-bimodule, with:
$$
\big((x,a)\cdot f\big)(z,c)=f((z,c)(x,a)),\qquad \big(f\cdot(x,a)\big)(z,c)=f((x,a)(z,c)),
$$
or if $f=(g,h)\in V^\star\times M^\star$, $(x,a)\cdot(g,h)=(x\cdot g+a\cdot h,x\cdot h)$, $(g,h)\cdot(x,a)=(g\cdot x+h\cdot a,h\cdot x)$.

This defines a structure of algebra on the space $\tilde{V}=W\times W^\star$, namely:
$$
\tilde{q}((x,a,g,h),(x',a',g',h'))=(xx',xa'+ax',x\cdot g'+g\cdot x'+a\cdot h'+h\cdot a',x\cdot h'+h\cdot x'),
$$
and a non degenerated symmetric bilinear form $\tilde b$,
$$
\tilde{b}((x,a,g,h),(x',a',g',h'))=g(x')+h(a')+g'(x)+h'(a).
$$
Now $(\tilde{V},\tilde{b},\tilde{q})$ is a quadratic associative algebra, it is the double semi-direct product of $(V,q)$ by the bimodule $M$. As usual, $\tilde{Q}$ is associated to $\tilde{q}$, after a shifting of degree.\\





Let now $c:V^k\rightarrow M$, $k$-linear, with degree $|c|=2-k$, and  identify $C$ with $\tilde{C}:\tilde{V}[1]^k\rightarrow\tilde{V}[1]$, by putting:
$$
\tilde{C}((x_1,a_1,g_1,h_1),\ldots,(x_k,a_k,g_k,h_k))=(0,C(x_1,\ldots,x_k),\sum_{j}C_j(x_1,\ldots,h_j,\ldots,x_k),0),
$$
where $C_j(x_1,\ldots,h_j,\ldots,x_k)=h_j(C(x_{j+1},\ldots,x_k,\cdot,x_1,\ldots,x_{j-1}))\in V^\star$. A direct computation shows that $\tilde{C}$ is $\tilde{B}$-quadratic. 
A direct computation gives:
$$
d_P\tilde{Q}=[\tilde{Q},\tilde{C}]=\widetilde{[Q,C]}=\widetilde{(d_Hc)[1]},
$$
where $d_H$ is the Hochschild coboundary operator on the bimodule $M$. 

\begin{prop}

\

Let $(V,q)$ be an associative algebra, and $c\mapsto\Omega_{\tilde{C}}$ the map associating to any multilinear mapping $c$ from $V^k$ into $M$, with degree $2-k$, the cyclic form $\Omega_{\tilde{C}}$. Then this map is a complex morphism between the Hochschild cohomology for the $(V,q)$ bimodule $M$ and the Pinczon cohomology of cyclic forms $\mathcal C(\tilde{V})$ on $\tilde{V}$.\\
\end{prop}


\section{Commutative Pinczon algebras}


\

\subsection{Commutative quadratic algebras}

\

Consider a quadratic associative algebra $(V,b,q)$, but suppose now $q$ is commutative
. Consider, as above the corresponding coderivation $Q$. It is now anticommutative, with degree 1, and seen as a map $\underline\otimes^2V[1]\rightarrow V[1]$, where $\underline\otimes^2V[1]$ is the quotient of $\otimes^2V[1]$ by the $1,1$ shuffle product $sh_{1,1}(x_1,x_2)=x_1\otimes x_2+x_2\otimes x_1$.

Recall that a $p,q$ shuffle $\sigma$ is a permutation $\sigma\in\mathfrak S_{p+q}$ such that $\sigma(1)<\ldots<\sigma(p)$ and $\sigma(p+1)<\ldots<\sigma(p+q)$. Denote $Sh(p,q)$ the set of all such shuffles. Then the $p,q$ shuffle product on $\otimes^+V[1]$ is
$$
sh_{p,q}(x_{[1,p]},x_{[p+1,p+q]})=\sum_{\sigma\in Sh(p,q)}x_{\sigma^{-1}(1)}\otimes\ldots\otimes x_{\sigma^{-1}(p+q)}.
$$

Denote $\underline{\otimes}^n(V[1])$ the quotient of $\otimes^nV[1]$ by the sum of all the image of the maps $sh_{p,n-p}$ ($0<p<n$), and $\underline{x}_{[1,n]}=x_1\underline{\otimes}\ldots\underline{\otimes}x_n$ the class of $x_1\otimes\ldots\otimes x_n$, where the $x_i$ belong to $V[1]$. Finally, $\underline{\otimes}^+(V[1])$ is the sum of all $\underline{\otimes}^n(V[1])$ ($n>0$).\\

On $\underline{\otimes}^+V[1]$, there is a Lie cobracket $\delta$, defined by
$$
\delta(\underline{x}_{[1,n]})=\sum_{j=1}^{n-1}\underline{x}_{[1,j]}\bigotimes \underline{x}_{[j+1,n]}-\underline{x}_{[j+1,n]}\bigotimes \underline{x}_{[1,j]}.
$$
In fact $\delta$ is well defined on the quotient and any coderivation $Q$ of $\delta$ is characterized by its Taylor expansion $Q=\sum_k Q_k$ where each $Q_k$ is a linear map from $\underline{\otimes}^kV[1]$ into $V[1]$ (see for instance \cite{AAC1,BGHHW}).\\

\begin{defn}
A structure of $C_\infty$ algebra, or up to homotopy commutative algebra, on $V$ is a degree 1 coderivation $Q$ of $\delta$, on $\underline{\otimes}^+V[1]$, such that $[Q,Q]=0$.\\
\end{defn}

Associated to the notion of vanishing on shuffle products mapping $Q$, there is a notion of vanishing on shuffle product forms $\Omega$.

\begin{defn} 
A $(k+1)$-linear cyclic form $\Omega$ on the vector space $V[1]$ is vanishing on shuffle product if, for any $y$, $(x_1,\ldots,x_k)\mapsto \Omega(x_1,\ldots,x_k,y)$ is vanishing on shuffle product. Denote $\mathcal C_{vsp}(V)$ the space of cyclic, vanishing on shuffle product multilinear forms on $V[1]$.\\
\end{defn}

\begin{prop}
Suppose $\{~,~\}$ is a Pinczon bracket on the space $\mathcal C(V)$ of cyclic multilinear forms on $V[1]$. Then $\mathcal C_{vsp}(V)$ is a Lie subalgebra of $(\mathcal C(V),\{~,~\})$.\\
\end{prop}

\begin{proof}
In fact, the Pinczon bracket defines a non degenerate form $b$ on $V$, thus $B$ on $V[1]$, any form $\Omega$ can be written as $\Omega=\Omega_Q$, with
$$
\Omega_Q(x_1,\ldots,x_{k+1})=B(Q(x_1,\ldots,x_k),x_{k+1}).
$$

Therefore $\Omega$ is in $\mathcal C_{vsp}(V)$ if and only if $Q$ is vanishing on shuffle products. Now, in \cite{AAC2} it is shown that if $Q$, $Q'$ are vanishing on shuffle products, then $[Q,Q']$ is also vanishing on shuffle products. This proves the proposition.\\ 
\end{proof}

\begin{defn}
A commutative Pinczon algebra $(\mathcal C(V),\{~,~\},\Omega)$ is a Pinczon bracket $\{~,~\}$ on $\mathcal C(V)$, and a degree 3 form $\Omega\in \mathcal C_{vsp}(V)$, such that $\{\Omega,\Omega\}=0$.\\
\end{defn}

As for associative algebra, a commutative Pinczon algebra with $\Omega$ trilinear is simply a quadratic commutative algebra $(V,q,b)$.\\

\begin{prop}
Let $(\mathcal C(V),\{~,~\},\Omega)$ be an commutative Pinczon algebra. Write $\Omega=\Omega_Q$, $Q=\sum_kQ_k$, with $Q_k:\underline{\otimes}^kV[1]\rightarrow V[1]$. Then $Q$ is a structure of $C_\infty$ algebra on $V$, and each $Q_k$ is $B$-quadratic for the bilinear form $B$ coming from the bracket.

Conversely, if $(V,b)$ is a vector space with a non degenerated symmetric bilinear form, any $B$-quadratic structure $Q$ of $C_\infty$ algebra on $V$ defines an unique structure of commutative Pinczon algebra on $V$.\\
\end{prop}


\subsection{(Bi)modules and Harrison cohomology}

\

Any module $M$ on a commutative algebra $(V,q)$ is a bimodule where right and left action are coinciding. Let now $(V,q)$ be a commutative algebra, and $M$ a $(V,q)$-module. Repeat the preceding construction of 
the quadratic associative algebra $(\tilde{V},\tilde{b},\tilde{q})$. Now $(\tilde{V},\tilde{b},\tilde{q})$ is commutative.\\





As above, look now for a $k$-linear mapping $c$ from $V^k$ into $M$, with degree $2-k$ and vanishing on shuffle products, denote $\tilde{C}$ the corresponding map $\tilde{V}^k\rightarrow \tilde{V}$, with degree 1,
We saw that $d_P\tilde{C}=[\tilde{Q},\tilde{C}]=\widetilde{[Q,C]}=\widetilde{d_Hc[1]}$. If we restrict this to vansihing on shuffle products map $c$, this is the Harrison coboundary $\widetilde{d_{Ha}c[1]}$.\\ 

\begin{prop}
Let $(V,q)$ be a commutative algebra, and $c\mapsto\Omega_{\tilde{C}}$ the map associating to any multilinear mapping $c$ from $V^k$ into $M$, with degree $2-k$, the cyclic form $\Omega_{\tilde{C}}$. Then this map is a complex morphism between the Harrison cohomology for the $(V,q)$ bimodule $M$ and the Pinczon cohomology of cyclic forms $\mathcal C(\tilde{V})$ on $\tilde{V}$.\\
\end{prop}


\section{Pinczon Lie algebras}


\

\subsection{Quadratic Lie algebras}

\

Suppose now $(V,q)$ is a (graded) Lie algebra. Then the corresponding Bar resolution consists in replacing the space $\otimes^+V[1]$ of tensors by the subspace $S^+(V[1])$ of symmetric tensors, spanned by the symmetric products $x_1\cdot\ldots\cdot x_k=\sum_{\sigma\in\mathfrak S_k}x_{\sigma(1)}\otimes\ldots\otimes x_{\sigma(k)}$. Then, the natural comultiplication $\Delta$ on $S^+(V[1])$ is:
$$
\Delta(x_1\cdot\ldots\cdot x_k)=\sum_{\begin{smallmatrix}I\sqcup J=[1,k]\\
0<\#I<k\end{smallmatrix}} x_{\cdot I}\bigotimes x_{\cdot J},
$$
where, if $I=\{i_1<\ldots<i_r\}$, $x_{\cdot I}$ means $x_{i_1}\cdot\ldots\cdot x_{i_r}$.

As above, any coderivation $Q$ of the comultiplication $\Delta$ is characterized by its Taylor coefficients $Q_k:S^k(V[1])\rightarrow V[1]$. The bracket of two such coderivations $Q$, $Q'$ becomes:
$$
[Q,Q'](x_1\cdot\ldots\cdot x_{k+k'-1})=\hskip-0.5cm\sum_{\begin{smallmatrix}I\sqcup J=[1,k+k'-1]\\
\#J=k'\end{smallmatrix}}\hskip-0.5cm Q(Q'(x_{\cdot J})\cdot x_{\cdot I})-\hskip-0.5cm\sum_{\begin{smallmatrix}I\sqcup J=[1,k+k'-1]\\
\#I=k\end{smallmatrix}}\hskip-0.5cm Q'(Q(x_{\cdot I})\cdot x_{\cdot J}),
$$

Now to the Lie bracket $q$ on $V$ is associated a map $Q:S^2(V[1])\rightarrow V[1]$, thus a degree 1 coderivation, still denoted $Q$, of $\Delta$ and the Jacobi identity for $q$ is equivalent to the relation $[Q,Q]=0$. The corresponding cohomology is the Chevalley cohomology.\\

For any $k$, consider $S^k(V[1])$ as a subspace of $\otimes^k V[1]$, with a projection $Sym:\otimes^k V[1]\rightarrow S^k(V[1])$:
$$
Sym(x_1\otimes\ldots\otimes x_k)=x_1\cdot\ldots\cdot x_k.
$$

Denote $L(\otimes^kV[1],V[1])$ the of linear maps $Q:\otimes^kV[1]\rightarrow V[1]$. By restriction to $S^k(V[1])$, $Q$ defines a symmetric map, denote it $Q^{Sym}$. Then $Q^{Sym}=Q\circ Sym$, and the space $L(S^k(V[1]),V[1])$ of symmetric maps is a quotient of $L(\otimes^kV[1],V[1])$.\\

Similarly, the restriction of $\Omega\in \mathcal C^{k+1}$ is a symmetric form, denote it $\Omega^{Sym}$: $\Omega^{Sym}=\Omega\circ Sym$.
as above, the space $\mathcal C|_S(V)$ of the restrictions of cyclic multilinear forms to $S^+(V[1])$ is a quotient of $\mathcal C(V)$.\\

Suppose there is a Pinczon bracket $\{~,~\}$ on $\mathcal C(V)$, then any $\Omega$ can be written $\Omega=\Omega_Q$, with $Q$ $k$-linear and $B$-quadratic. But, any element $\sigma$ in $\mathfrak S_{k+1}$ can be written in an unique way as a product $\tau\circ\rho$ where $\tau$ is in $\mathfrak S_k$, viewed as a subgroup of $\mathfrak S_{k+1}$ and $\rho$ in $Cycl$. Therefore, with our notation,
$$
\Omega^{Sym}=(\Omega_Q)^{Sym}=(\Omega_{Q^{Sym}})^{Cycl}=(k+1)\Omega_{Q^{Sym}}.
$$

\begin{prop}
The bracket defined on $L(\otimes^+V[1],V[1])$ by the commutator of coderivations in $(\otimes^+V[1],\Delta)$, induces a well defined bracket on 
 $L(S^+(V[1]),V[1])$, this bracket is the above commutator of coderivations in $(S^+(V[1]),\Delta)$:
$$
\left[Q^{Sym},(Q')^{Sym}\right]=[Q,Q']^{Sym}.
$$

Any Pinczon braket $\{~,~\}$ on $\mathcal C(V)$ induces a well defined bracket on the quotient $\mathcal C|_S(V)$, this bracket denoted $\{~,~\}|$ is:
$$
\left\{\Omega^{Sym},(\Omega')^{Sym}\right\}|=\left(\{\Omega,\Omega'\}\right)^{Sym}=\frac{k+k'}{(k+1)!(k'+1)!} \sum_i\iota_{e_i}\left(\Omega^{Sym}\right)\cdot\iota_{e'_i}\left((\Omega')^{Sym}\right).
$$
\end{prop}

\begin{proof}
The first assertion is a simple computation:
$$\aligned
(Q\circ Q')(x_{[1,k+k'-1]})&=\sum_{r=1}^{k-1}Q(x_{[1,r-1]}\otimes Q'(x_{[r,r+k'-1]})\otimes x_{[r+k',k+k'-1]}).
\endaligned
$$

Then, 
$$
(Q\circ Q')^{Sym}(x_{\cdot[1,k+k'-1]})=
\sum_{\begin{smallmatrix}r\\ \sigma\in\mathfrak S_{k+k'-1}\end{smallmatrix}}Q(x_{\sigma^{-1}([1,r-1])},Q'(x_{\sigma^{-1}([r,r+k'-1])}),x_{\sigma^{-1}([r+k',k+k'-1])})\\
$$

Fix $r$ and $\sigma$, put $J=\sigma^{-1}([r,r+k'-1])=\{j_1<\ldots<j_{k'}\}$, and $j_{\tau^{-1}(t)}=\sigma^{-1}(r+t-1)$ ($1\leq t\leq k'$). This define $\sigma\in\mathfrak S_{k'}$. Similarly, put $I=[1,k+k'-1]\setminus J$, $\{0\}\cup I=\{i_1<\ldots<i_k\}$ and define $\rho\in\mathfrak S_k$ by:
$$
i_{\rho^{-1}(t)}=\sigma^{-1}(t)\quad (1\leq t\leq r-1),\quad i_{\rho^{-1}(r)}=0, \quad i_{\rho^{-1}(t)}=\sigma^{-1}(t-1)\quad (r+1\leq t\leq k).
$$
The correspondence $(r,\sigma)\mapsto(r,\tau,\rho)$ is one-to-one and, suming up, we get:
$$
(Q\circ Q')^{Sym}(x_{\cdot[1,k+k'-1]})=\sum_{\begin{smallmatrix}I\sqcup J=[1,k+k'-1]\\
\#J=k'\end{smallmatrix}}Q^{Sym}((Q')^{Sym}(x_{\cdot J})\cdot x_{\cdot I}).
$$
The commutator of coderivation of $(S^+(V[1]),\Delta)$ is the quotient bracket:
$$
[Q,Q']^{Sym}=\left[Q^{Sym},(Q')^{Sym}\right].
$$

Let now $\{~,~\}$ be a Pinczon bracket on the space $\mathcal C(V)$. Thus any cyclic form $\Omega$ is written as $\Omega_Q$. Then:
$$
\{\Omega,\Omega'\}^{Sym}=\{\Omega_Q,\Omega_{Q'}\}^{Sym}=\Omega_{[Q,Q']}^{Sym}=(k+k')\Omega_{[Q,Q']^{Sym}}=(k+k')\Omega_{\left[Q^{Sym},(Q')^{Sym}\right]}.
$$
The last bracket is the commutator of coderivations in $S^+(V[1])$, thus
$$\aligned
\{\Omega,\Omega'\}^{Sym}(x_{\cdot[1,k+k']})&=
(k+k')\sum_{\begin{smallmatrix}I\sqcup J=[1,k+k']\\ \#J=k'\end{smallmatrix}}B(Q^{Sym}(x_{\cdot I}),(Q')^{Sym}(x_{\cdot J})).
\endaligned
$$

On the other hand,
$$
\iota_{e_i}\left(\Omega_Q^{Sym}\right)(x_{\cdot[1,k]})=(k+1)B(Q^{Sym}(x_{\cdot[1,k]}),e_i)=(k+1)\left(\iota_{e_i}\Omega\right)^{Sym}(x_{\cdot[1,k]}).
$$
Therefore
$$\aligned
\sum_i\iota_{e_i}\left(\Omega_Q^{Sym}\right)\cdot\iota_{e'_i}\left(\Omega_{Q'}^{Sym}\right)(x_{\cdot[1,k+k']})&=\sum_i\left(\iota_{e_i}\left(\Omega_Q^{Sym}\right)\otimes\iota_{e'_i}\left(\Omega_{Q'}^{Sym}\right)\right)^{Sym}(x_{\cdot[1,k+k']})\\
&\hskip-1.8cm=(k+1)!(k'+1)!\sum_{\begin{smallmatrix}I\sqcup J=[1,k+k']\\ \#J=k'\end{smallmatrix}}B(Q^{Sym}(x_{\cdot I}),(Q')^{Sym}(x_{\cdot J}))\\
\endaligned
$$
This achieves the proof.
\end{proof}

Explicitely, if the forms $\Omega$ and $\Omega'$ are symmetric, 
$\Omega^{Sym}=(k+1)!\Omega$ and the Pinczon bracket on the quotient becomes the bracket defined in \cite{PU,DPU}:
$$
\{\Omega,\Omega'\}|=(k+k')\sum_i\iota_{e_i}\Omega\cdot\iota_{e'_i}\Omega'.
$$

\subsection{Quadratic $L_\infty$ algebras}

\

As above,

\begin{defn}
A Pinczon Lie algebra $(\mathcal C(V),\{~,~\},\Omega)$ is a Pinczon bracket $\{~,~\}$ on $\mathcal C(V)$, and a degree 3 symmetric form $\Omega\in \mathcal C(V)^{Sym}$, such that $\{\Omega,\Omega\}=0$.\\
\end{defn}

Since a structure of $L_\infty$ algebra (or Lie algebra up to homotopy) on $V$ is a degree 1 coderivation $Q$ of $(S^+(V[1]),\Delta)$, such that the commutator $[Q,Q]$ vanishes. With the preceding computations, a Pinczon Lie algebra is in fact a quadratic $L_\infty$ algebra.\\

\begin{prop}
Let $(\mathcal C(V),\{~,~\},\Omega)$ be a Pinczon Lie algebra. Write $\Omega=\Omega_Q$, $Q=\sum_kQ_k$, with $Q_k:S^k(V[1])\rightarrow V[1]$. Then $Q$ is a structure of $L_\infty$ algebra on $V$, and each $Q_k$ is $B$-quadratic for the bilinear form $B$ coming from the bracket.

Conversely, if $(V,b)$ is a vector space with a non degenerated symmetric bilinear form, any $B$-quadratic structure $Q$ of $L_\infty$ algebra on $V$ defines an unique structure of commutative Pinczon algebra on $V$.\\
\end{prop}

\subsection{Modules and Chevalley cohomology}

\

Let $(V,q)$ be a Lie algebra, and $M$ a $(V,q)$-module. To refind the corresponding Chevalley coboundary operator $d_{Ch}$, build, as above, the double semidirect product of $V$ by $M$.\\

First, consider the semidirect product $W=V\rtimes M$, that is the vector space $V\oplus M$, equipped with the Lie bracket $q_W((x,a),(y,b))=([x,y],x\cdot b-y\cdot a)$; its dual $W^\star$ is also a $(W,q_W)$-module.

Then $\tilde{V}=W\times V\times W^\star$ is a quadratic Lie algebra for the bracket
$$
\tilde{q}((x,a,g,h),(x',a',g',h'))=([x,x'],x\cdot a'-x'\cdot a,x\cdot g'-x'\cdot g+a\cdot h'-a'\cdot h,x\cdot h'-x'\cdot h),
$$
and the non degenerated symmetric bilinear form $\tilde b$
$$
\tilde{b}((x,a,g,h),(x',a',g',h'))=g(x')+h(a')+g'(x)+h'(a).
$$
A direct computation (see \cite{BB,MR}) shows that $(\tilde{V},\tilde{b},\tilde{q})$ is a quadratic Lie algebra, the double semidirect product of $(V,q)$ by its module $M$.\\




Look now for a  skew-symmetric $k$-linear mapping $c$ from $V^k$ into $M$, with degree $|c|=2-k$. Associate to it the mapping:
$$
\tilde{C}((x_1,a_1,g_1,h_1),\ldots,(x_k,a_k,f_k,g_k))=(0,C(x_1,\ldots,x_k),\sum_{j=1}^kC_j(x_1,\ldots,h_j,\ldots,x_k),0).
$$
Clearly, $\tilde{C}$ is totaly symmetric from $\tilde{V}[1]^k$ into $\tilde{V}[1]$. More precisely $\widetilde{C^{Sym}}={\tilde{C}}^{Sym}$.

Then, 
$$
[\tilde{Q},\tilde{C}]^{Sym}={\widetilde{[Q,C]}}^{Sym}=\widetilde{[Q,C]^{Sym}}=\widetilde{d_{Ch}c[1]}.
$$
Now, if $d_P$ is the Pinczon coboundary operator, defined by:
$$
\{\Omega_{\tilde{Q}},\Omega_{\tilde{C}}\}|=\Omega_{d_P\tilde{C}},
$$
this can be written $d_P\tilde{C}=(2+k)\widetilde{d_{Ch}c[1]}$.\\

\begin{prop}
Let $(V,q)$ be a Lie algebra, and $c\mapsto\Omega_{\tilde{C}}$ the map associating to any multilinear skewsymmetric mapping $c$ from $V^k$ into $M$, with degree $2-k$, the symmetric form $\Omega_{\tilde{C}}$. Then this map is a complex morphism between the Chevalley cohomology for the $(V,q)$ module $M$ and the Pinczon cohomology of symmetric forms $\mathcal C(\tilde{V})|_S$ on $\tilde{V}$.\\
\end{prop}


\section{Pinczon pre-Lie algebras}


\

\subsection{Quadratic pre-Lie algebras}

\

A left pre-Lie algebra $(V,q)$ is a (graded) vector space with a product $q$ such that:
$$
q(x,q(y,z))-q(q(x,y),z)= q(y,q(x,z))-q(q(y,x),z).
$$
Then the bracket $[x,y]=q(x,y)-q(y,x)$ is a Lie bracket. Remark that any associative algebra $(V,q)$ is a pre-Lie algebra.

A vector space $M$ is a left $(V,q)$-module for the linear map $x\otimes a\mapsto x\cdot a$, if
$$
q(x,y)\cdot a-x\cdot(y\cdot a)= q(y,x)\cdot a-y\cdot(x\cdot a)\quad (a\in M,~ x\in V).
$$
A left module is a bi-module, if there is a linear map $a\otimes x\mapsto a*x$ such that:
$$
(a*x)*y-a*q(x,y)= (x\cdot a)*y-x\cdot(a*y)\quad (a\in M,~ x,y\in V).
$$
Then a direct computation proves\\

\begin{lem}
Let $(V,q)$ a left pre-Lie algebra, then
\begin{itemize}
\item[1.] The dual $V^\star$ of $V$ is a left $(V,q)$-module, for: 
$(x\cdot f)(y)= -f(q(x,y))$.
\item[2.] Let $M$ be a $(V,q)$ bi-module then $W=V\rtimes M=(V\oplus M,q_W)$, where:
$$
q_W(x+a,y+b)= q(x,y)+x\cdot b+a*y
$$
is a left pre-Lie algebra, the semi-direct product of $V$ by $M$.
\end{itemize}
\end{lem}

Chapoton and Livernet define in \cite{CL} the notion of pre-$L_\infty$-algebra. If $V$ is a graded vector space, we consider the space $S(V[1])\otimes V([1])$, generated by the tensors
$$
x_{[1,k]}\otimes y= x_1\cdot\ldots\cdot x_k\otimes y.
$$

Put $P^k=S^k(V[1])\otimes V([1])$, and $P=\sum_{k\geq 0}P^k$. On $P$, the coproduct $\Delta$ is defined by $\Delta(1\otimes y)=0$ and:
$$
\Delta(x_{[1,k]}\otimes y)=
\sum_{\begin{smallmatrix}1\leq j\leq k\\ I\sqcup J=[1,k]\setminus\{j\}\end{smallmatrix}} (x_I\otimes x_j)\bigotimes (x_J\otimes y).
$$

A linear map $Q:P^k\rightarrow P^0$ extends to a coderivation, still denoted $Q$ by:
$$
Q(x_{[1,n]}\otimes y)=\sum_{\begin{smallmatrix}I\sqcup J=[1,n]\\
\#J=k\end{smallmatrix}} x_I\otimes Q(x_J\otimes y)+\hskip-0.5cm
\sum_{\begin{smallmatrix}1\leq j\leq n\\ I\sqcup J=[1,n]\setminus\{j\}\\
\#J=k\end{smallmatrix}}x_I\cdot Q(x_J\otimes x_j)\otimes y.
$$

The commutator of two coderivations is a coderivation, and a map $q:V\otimes V\rightarrow V$ is a left pre-Lie product if and only if the structure equation $[Q,Q]=0$ holds for the corresponding map $Q:P^1\rightarrow P^0$. Thus:

\begin{defn}
A structure of pre-$L_\infty$ algebra on $V$ is a degree 1 coderivation $Q$ of $(P,\Delta)$, such that $[Q,Q]=0$.\\
\end{defn}

Now, a quadratic pre-Lie algebra $(V,b,q)$ is a pre-Lie algebra $(V,q)$ equipped with a symmetric, non degenerated, degree 0, bilinear form $b$ such that
$$
b(q(x,y),z)+b(y,q(x,z))= 0,\quad\text{ or }\quad B(Q(x,y),z)=B(Q(x,z),y).
$$
\begin{ex}\label{VV*}
Let $(V,q)$ be a left pre-Lie algebra, consider the pre-Lie algebra, semidirect product $W=V\rtimes V^\star$, with
$$
q_W(x+f,y+g)= q(x,y) + x\cdot g= q(x,y)-g(q(x,\cdot)).
$$
It is a quadratic pre-Lie algebra if we endow $W$ by the canonical symmetric, non degenerated form $b(x+f,y+g)=f(y)+g(x)$.\\
\end{ex}

Generalizing, let us say that a coderivation $Q=Q_0+Q_1+\ldots$ of $\Delta$ is $B$-quadratic if, for any $k$,
$$
B(Q_k(x_{[1,k]}\otimes y_1),y_2)= B(Q_k(x_{[1,k]}\otimes y_2),y_1).
$$
Then a direct computation gives:

\begin{lem}
Let $Q$ and $Q'$ be two $B$-quadratic coderivations of $\Delta$. Then $[Q,Q']$ is $B$-quadratic.\\
\end{lem}

A structure of quadratic pre-$L_\infty$ algebra on $(V,b)$ is thus a $B$-quadratic coderivation $Q$ of $\Delta$ such that $[Q,Q]=0$.\\


\subsection{Pinczon bracket for pre-Lie algebras}


\

To a $k+1$-linear coderivation $Q$, let us associate the form $\Omega_Q$:
$$
\Omega_{Q}(x_{[1,k]}\otimes y_1\cdot y_2)=B(Q(x_{[1,k]}\otimes y_1),y_2)+B(Q(x_{[1,k]}\otimes y_2),y_1).
$$

This form is separately symmetric in its $k$ first variables, and its 2 last variables. Define thus $\mathcal P_k$ the space of such forms, and $\mathcal P(V)=\sum_{k\geq 0}\mathcal P_k$. An element of $\mathcal P(V)$ is called a bi-symmetric form. Now it is possible to extend to $\mathcal P(V)$ the Pinczon bracket $\{~,~\}$ associated to $B$. Indeed, a direct computation shows:

\begin{lem}
Let $\mathfrak g$ be a Lie algebra and $A$ a commutative algebra which is a right $\mathfrak g$-module such that, for any $x\in \mathfrak g$, $f\mapsto f\cdot x$ is a derivation of $A$. Then the formula:
$$
[f\otimes x,g\otimes y]= fg\otimes [x,y]+(f\cdot y)g\otimes x-f(g\cdot x)\otimes y
$$
defines a Lie bracket on $A\otimes \mathfrak g$.\\ 
\end{lem}

Denote $\mathcal S(V)=(\mathcal C(V))^{Sym}$ the symmetric algebra of $V$, it is a commutative algebra for the symmetric product $\cdot$. On the other hand, by construction, $\mathcal C^2$ is a Lie algebra for the Pinczon bracket $\{~,~\}$, acting on $\mathcal S$ through $(\Omega,\alpha)\mapsto \{\Omega,\alpha\}$. Then the properties of the Pinczon bracket assure that $\mathcal S$ is a $(\mathcal S^2,\{~,~\})$-module and the action is a derivation of $\mathcal S$. Finally remark that $\mathcal P(V)=\mathcal S(V)\otimes \mathcal C^2$, therefore:

\begin{cor}
Let $(V,b)$ be a vector space with a symmetric, non degenerated bilinear form, then the space $\mathcal P(V)$ of bi-symmetric forms on $V$ is a Lie algebra for the Pinczon bracket:
$$
\{\Omega\otimes\alpha,\Omega'\otimes\alpha'\}=\Omega\cdot\{\alpha,\Omega'\}\otimes \alpha'+\Omega'\cdot\{\Omega,\alpha'\}\otimes\alpha+\Omega\cdot\Omega'\otimes\{\alpha,\alpha'\}.
$$
\end{cor}

This bracket is related to the commutator of $B$-preserving coderivations, since:

\begin{prop}
Suppose $b$ is a symmetric non degenrated form on $V$, and $Q$, $Q'$ two $B$-quadratic coderivations of $\Delta$. Consider the forms $\Omega_Q$, $\Omega_{Q'}$ in $\mathcal P(V)$, then
$$
\{\Omega_Q,\Omega_{Q'}\}=2\Omega_{[Q,Q']}.
$$ 
\end{prop}

\begin{proof}
Suppose $\Omega_Q=\beta\otimes\alpha\in\mathcal P^k$, and $\Omega_{Q'}=\beta'\otimes\alpha'\in\mathcal P^{k'}$, then
$$\aligned
\beta\cdot \beta'&\otimes\{\alpha,\alpha'\}(x_{[1,k+k']},y_{[1,2]}) =\\
&=\hskip-0.5cm \sum_{\begin{smallmatrix}i,~I\sqcup J=[1,k+k']\\ \#I=k\end{smallmatrix}}\big(\Omega_Q(x_I\otimes e_i\cdot y_1)\Omega_{Q'}(x_J\otimes e'_i\cdot y_2)
+\Omega_Q(x_I\otimes e_i\cdot y_2)\Omega_{Q'}(x_J\otimes e'_i\cdot y_1)\big)\\
&= 4\hskip -0.5cm \sum_{\begin{smallmatrix}I\sqcup J=[1,k+k']\\ \#I=k\end{smallmatrix}} -B(Q'(x_J\otimes Q(x_I\otimes y_1)), y_2)+B(Q(x_I\otimes Q'(x_J\otimes y_1)),y_2)
\endaligned
$$
On the other hand,
$$\aligned
\beta\cdot\{\alpha,\beta'\}&\otimes\alpha'(x_{[1,k+k']},y_{[1,2]}) =\\
& =\hskip-0.5cm \sum_{\begin{smallmatrix}j,I\sqcup J=[1,k+k']\setminus\{j\}\\ \#I=k\end{smallmatrix}}\sum_i\Omega_Q(x_I\otimes e_i\cdot x_j)\Omega_{Q'}(e'_i\cdot x_J\otimes y_1\cdot y_2)
\\
&= -4\hskip -0.5cm\sum_{\begin{smallmatrix}j,I\sqcup J=[1,k+k']\setminus\{j\}\\ \#I=k\end{smallmatrix}} B(Q'(Q(x_I\otimes x_j)\cdot x_J\otimes y_1),y_2).
\endaligned
$$
And similarly for the last term $\{\beta,\alpha'\}\cdot\beta'\otimes \alpha$ in the Pinczon bracket. Suming up, we get:
$$\aligned
\{\Omega_Q,\Omega_{Q'}\}(x_{[1,k+k']}\otimes y_{[1,2]})
&\eqKu 2B([Q,Q'](x_{[1,k+k']}\otimes y_1),y_2)+2B([Q,Q'](x_{[1,k+k']}\otimes y_2),y_1)\\
&\eqKu 2\Omega_{[Q,Q']}(x_{[1,k+k']}\otimes y_{[1,2]}).
\endaligned
$$
\end{proof}

As for the other sort of algebras,
\begin{defn}
A Pinczon pre-Lie algebra $(\mathcal P(V),\{~,~\},\Omega)$ is a Pinczon bracket $\{~,~\}$ on $\mathcal C(V)$, extended to $\mathcal P(V)$, and a degree 3 bi-symmetric form $\Omega\in \mathcal P(V)$, such that $\{\Omega,\Omega\}=0$.\\
\end{defn}

Since a structure of pre-$L_\infty$ algebra on $V$ is a degree 1 coderivation $Q$ of $(P,\Delta)$, such that the commutator $[Q,Q]$ vanishes. With the preceding computations, a Pinczon pre-Lie algebra is in fact a quadratic pre-$L_\infty$ algebra.\\

\begin{prop}
Let $(\mathcal P(V),\{~,~\},\Omega)$ be a Pinczon pre-Lie algebra. Write $\Omega=\Omega_Q$, and $Q=\sum_k Q_k$, with $Q_k:P^k\rightarrow P^0$. Then $Q$ is a pre-$L_\infty$ algebra structure on $V$, and each $Q_k$ is $B$-quadratic. 

Conversely, if $(V,b)$ is a vector space with a non degenerated symmetric bilinear form, any $B$-quadratic structure $Q$ of pre-$L_\infty$ algebra on $V$ defines an unique structure of Pinczon pre-Lie algebra on $V$.\\
\end{prop}


\subsection{Pinczon and pre-Lie algebra cohomologies}


\

Suppose $(V,q)$ is a left pre-Lie algebra and let $Q$ be the coderivation of $\Delta$ associated to $q$. Since $[Q,Q]=0$, $d:C\mapsto [Q,C]$ is a coboundary operator. 

Now, let $M$ be a $(V,q)$ bi-module. Let $W=V\rtimes M$ be the semi-direct product of $V$ by $M$. Any map $c:\wedge^k V\otimes V\rightarrow M$ can be naturally extended to a map, still denoted $c$, from $\wedge^k W\otimes W$ to $W$. If $C$ is the corresponding coderivation, the map $dc$, where the coderivation corresponding to $dc:\wedge^{k+1} W\otimes W\rightarrow W$ is $[Q,C]$, is the extension  of a map $d_{pL}c:\wedge^{k+1} V\otimes V\rightarrow M$. The operator $d$ is the pre-Lie coboundary operator. In an unpublished work, Ridha Chatbouri computed explicitely this operator:\\

\begin{prop}
The cohomology of the pre-Lie algebra $(V,q)$ is defined as follows. Let $M$ be a $(V,q)$ bi-module, the cohomology with value in $M$ is given by the following  operator: If $c:\wedge^k V\otimes V\rightarrow M$ is a $(k+1)$-cochain, with degree $|c|$, then $dc$ is explicitly:
$$\aligned
&(-1)^{|c|}d_{pL}c(x_0\wedge\ldots\wedge x_k\otimes y)=\\
&\sum_{i=0}^k(-1)^ic(x_0\wedge\ldots\hat{i}\ldots\wedge x_k\otimes x_i)*y
 -\sum_{i=0}^k(-1)^{i}c(x_0\wedge\ldots\hat{i}\ldots\wedge x_k\otimes q(x_i,y))\\
&+\sum_{i<j}(-1)^{i+j}c([x_i,x_j]\wedge x_0\wedge\ldots\hat{i}\hat{j}\ldots\wedge x_k\otimes y)
+\sum_{i=0}^k(-1)^ix_i\cdot c(x_0\wedge\ldots\hat{i}\ldots\wedge x_k\otimes y).
\endaligned
$$
\end{prop}

In \cite{Dz}, A. Dzhumaldil'daev defined a coboundary operator $d$ for right pre-Lie algebra, which is the same as the operator computed in the preceding proposition, modulo the change of side for pre-Lie axioms. Then he used this operator to compute a corresponding homology. The proof of the proposition is the inverse of the Dzhumaldil'daev proof.

\begin{rema}
A left $(V,q)$-module is nothing else that a $(V,[~,~])$-module. However the symmetry of a pre-Lie cochain differs of the symmetry of a Lie cochain. Thus the cocycles are not the same. For instance we consider $V=C^\infty_c(\R)$, with $q(f,g)=fg'$. Choose $M=V$ and $f\cdot g=q(f,g)$. Then $M$ is a left module. Put $c(f,g)=fg$. It is easy to verify it is a cocycle, but it is not skewsymmetric. In fact it is the coboundary of $f\mapsto b(f)$, with $(b(f))(t)=tf(t)$.\\
\end{rema}

More generally, if $Q$ is a coderivation of $\Delta$, which is a pre-$L_\infty$ structure, then the operator $C\mapsto[Q,C]$ is a coboundary operator. Let us call the corresponding cohomology the $(V,Q)$ cohomology. Then\\

\begin{cor}
Suppose $(V,b,q)$ is a quadratic pre-Lie algebra or, more generally, $(\mathcal P(V),\{~,~\},\Omega)$ a Pinczon pre-Lie algebra. Then the operator $d:\mathcal P(V)\rightarrow \mathcal P(V)$ defined by $d\Omega=\{\Omega_Q,\Omega\}$ is a coboundary operator and the corresponding cohomology coincides with the $(V,Q)$-cohomology.\\
\end{cor}


\section{A natural example}


\

Recall that the infinitesimal deformations of an associative (resp. Lie, resp. pre-Lie) algebra $(V,q)$ are described by the corresponding second cohomology group (Hochschild, pre-Lie or Chevalley) of $V$ (see \cite{G,NR}). Indeed, putting $q_t=q+tc$, with $t^2=0$, the associativity, Jacobi or pre-Lie relation are respectively equivalent to $[Q,C]=0$. Therefore, if $t^2=0$, $(V+tV,q_t)$, is an associative, Lie or pre-Lie algebra if and only if $c$ is respectively a Hochschild, a Chevalley or a pre-Lie cocycle.

Such a deformation is trivial if there is a linear map $a$ such that $\varphi_t(x)=x+ta(x)$ satisfies $q_t(\varphi_t(x),\varphi_t(y))=\varphi_t(q(x,y))$. With $t^2=0$, these conditions are equivalent to $c=d_Ha$, resp. $c=d_{Ch}a$, $c=d_{pL}a$. If the only infinitesimal deformations are the trivial ones, that is if the second Hochschild cohomology group $H^2((V,q))$ vanishes, we say that the corresponding structure is infinitesimally rigid.\\

Suppose now $(V,q)$ is an associative algebra. Therefore it is a pre-Lie algebra for the multiplication $q$, and a Lie algebra for the bracket $[x,y]=q(x,y)-q(y,x)$. Some of these structures can be rigid, and other can be not rigid. Let us study here the natural associative algebra $M_n(\mathbb K)$ of $n\times n$ matrices on a characteristic zero field $\mathbb K$. Denote $\mathfrak{gl}_n(\mathbb K)$ the corresponding Lie algebra, and $(M_n,q)$ the pre-Lie algebra. Remark that $\mathfrak{gl}_n(\mathbb K)$ is a direct product: $\mathfrak{sl}_n(\mathbb K)\oplus \mathbb K id$, where $\mathfrak{sl}_n(\mathbb K)$ is the space of traceless matrices. Put $f(x)=\frac{1}{n}\tr(x)id$. Recall first some well-known results (see \cite{CH} Chap IX.7, and use Hochschild-Serre sequence and Whitehead Lemmas \cite{J}):\\

\begin{prop}
Let $M$ be a $M_n(\mathbb K)$ bi-module, then $H^k(M_n(\mathbb K),M)=0$ for $k>0$. Especially, if $M=M_n(\mathbb K)$, 
$$
H^0(M_n(\mathbb K))=\mathbb K id,\quad H^k(M_n(\mathbb K))=0~\text{ for }~k>0.
$$

Let $M$ be a $\mathfrak{gl}_n(\mathbb K)$-module, and $k>0$, then:
$$
H^k(\mathfrak{gl}_n(\mathbb K),M)\simeq H^k(\mathfrak{sl}_n(\mathbb K),\mathbb K)\otimes M^{\mathfrak{gl}_n(\mathbb K)}\oplus H^{k-1}(\mathfrak{sl}_n(\mathbb K),\mathbb K)\otimes \left( H^1(\mathbb K id,M)\right)^{\mathfrak{gl}_n(\mathbb K)}.
$$
Especially, if $M=\mathfrak{gl}_n(\mathbb K)$, and $\tilde{f}$ is the class of the projection $f$, 
$$
H^0(\mathfrak{gl}_n(\mathbb K))=\mathbb K id,\qquad
H^1(\mathfrak{gl}_n(\mathbb K))=\mathbb K\tilde{f},\qquad H^2(\mathfrak{gl}_n(\mathbb K))=0.
$$
\end{prop}

Especially, $M_n(\mathbb K)$ and $\mathfrak{gl}_n(\mathbb K)$ are rigid. However the pre-Lie algebra $V$ is not infinitesimally rigid:

\begin{lem} 
The second cohomology group $H^2(M_n,q)$ of the pre-Lie algebra $(M_n,q)$ is not vanishing.\\
\end{lem}

\begin{proof}
Following \cite{Dz}, for any matrix $a$, consider $c_a(x,y)=\frac{1}{n} \tr(x)[y,a]$. Then:
$$
d_Hc_a(x,y,z)=\frac{1}{n}(\tr(y)x-\tr(xy)id+\tr(x)y)[z,a]=(d_Hf)(x,y)[z,a].
$$
Then $d_{pL}c_a(x,y,z)=d_Hc_a(x,y,z)-d_Hc_a(y,x,z)=0$.
But remark that $Z^2(M_n(\mathbb K))=B^2(M_n(\mathbb K))=B^2(M_n,q)$. Thus, if $a$ is not a scalar matrix, there is $z$ such that $[z,a]\neq0$, and $d_Hc_a(e_{12},e_{21},z)=-[z,a]\neq0$, $c_a$ is not in $B^2(M_n,q)$.\\
\end{proof}

Let $(x,y)\mapsto q(x,y)$ any algebraic structure on a space $V$ and $c_k:V\otimes V\rightarrow V$ ($k\geq 1$). If for any $t$, $q_t=q +\sum_{k>0} t^kc_k$ endows $V\otimes \mathbb K[[t]]$ with the same sort of structure, we say that $q_t$ is a formal deformation of $q$. If there is $\varphi_t=id+\sum_{k>0}t^k b_k$, such that $q_t(\varphi_t(x),\varphi_t(y))=\varphi_t(q(x,y))$, we say that $q_t$ is trivial. If there is a non trivial deformation $q_t$, we say that $(V,q)$ is not rigid. If moreover $c_k=0$ for any $k>1$, and $c_1\notin B^2(V,q)$, we say that $q_t=q+tc_1$ is a true deformation at order 1 of $(V,q)$.

Using induction on $k$, it is easy to prove that an infinetisimally rigid structure is rigid. Of course the existence of a true deformation at order 1 implies that $(V,q)$ is not rigid.\\

\begin{cor}
The algebras $M_n(\mathbb K)$, $\mathfrak{gl}_n(\mathbb K)$ are rigid. The pre-Lie algebra $(M_n,q)$ admits true deformations at order 1.\\
\end{cor}

\begin{proof}
The first points are the previous results. Consider now the pre-Lie algebra $(M_n,q)$. A direct computation shows that $q_t=q+tc_a$ is a true deformation at order 1 of $(V,q)$ (see also \cite{Dz}).\\
\end{proof}

Suppose now $\mathbb K=\mathbb C$. Then $x\mapsto\tr(x^2)$, and $x\mapsto(\tr(x))^2$ generate the space of $\mathfrak{gl}_n(\mathbb C)$-invariant degree 2 polynomial functions on $\mathfrak{gl}_n(\mathbb C)$ (see \cite{W}). Any symmetric invariant bilinear form $b$ on $\mathfrak{gl}_n(\mathbb C)$ can be written:
$$
b(x,y)=\alpha~\tr(xy)+\beta~\tr(x)\tr(y)=\alpha\tr(x(y-f(y)))+(\frac{1}{n}\alpha+\beta)\tr(x)\tr(y).
$$
Thus, $b$ is non degenerate, if and only if $\alpha(\alpha+n\beta)\neq0$.

Observe now that $(x,y)\mapsto\tr(xy)$ is $M_n(\mathbb C)$-invariant, but $(x,y)\mapsto\tr(x)\tr(y)$ is not $M_n(\mathbb C)$-invariant. Thus the non degenerated bilinear form $(x,y)\mapsto\tr(xy)$ generates the cone of $M_n(\mathbb C)$-invariant symmetric bilinear forms.

Similarly, if $b$ is $(M_n,q)$-invariant, it is $\mathfrak{gl}_n(\mathbb C)$-invariant, thus $b(x,y)=\alpha\tr(xy)+\beta\tr(x)\tr(y)$, and the invariance relation reads: 
$$
\alpha\tr(xyz+yxz)+\beta(\tr(xy)\tr(z)+\tr(y)\tr(xz))=0.
$$
Choosing for instance $x=y=z=e_{11}$, we get $2(\alpha+\beta)=0$, then $x=y=z=id$ gives $2n(\alpha+n\beta)=0$, the unique $(M_n,q)$-invariant symmetric bilinear form is identically zero. Summarizing,

\begin{lem}
\begin{itemize}
\item[1.] The associative algebra $M_n(\mathbb C)$ is a quadratic algebra for the 1-dimensional cone of symmetric bilinear forms $\alpha\tr(xy)$,
\item[2.] The Lie agebra $\mathfrak{gl}_n(\mathbb C)$ is a quadratic Lie algebra for the 2-dimensional cone of symmetric bilinear forms $\alpha\tr(xy)+\beta\tr(x)\tr(y)$ ($\alpha(\alpha+n\beta)\neq0$),
\item[3.] The pre-Lie algebra $(M_n,q)$ is not a quadratic pre-Lie algebra.\\
\end{itemize}
\end{lem}

Recall that in Example \ref{VV*}, we saw that the space $M_n(\mathbb C)\oplus M_n(\mathbb C)$ is a quadratic pre-lie algebra for the product and the bilinear form:
$$
q_W(x_1+x_2,y_1+y_2)=x_1y_1-y_2x_1,\qquad b(x_1+x_2,y_1+y_2)=\tr(x_2y_1+x_1y_2).
$$





\end{document}